\documentclass[11pt,reqno]{amsart}
\usepackage{graphicx,tikz}
\usepackage{amsfonts,amsmath,amssymb}
\usetikzlibrary{matrix,arrows}
\newcommand{\cx}{{\mathbb{C}}}

\newcommand{\G}{\Gamma}

\newcommand{\D}{\mathbb D}
\newcommand{\C}{\mathbb C}

\newcommand{\N}{\mathbb N}

\newtheorem{theorem}{Theorem}[section]
\newtheorem{lemma}[theorem]{Lemma}
\newtheorem{prop}[theorem]{Proposition}
\newtheorem{cor}[theorem]{Corollary}

\theoremstyle{definition}

\begin{document}
\title{Dicritical singularities  and laminar currents on  Levi-flat hypersurfaces }
\author{Sergey Pinchuk*, Rasul Shafikov** and Alexandre Sukhov***}
\begin{abstract}
We establish an effective criterion for a dicritical singularity of a real analytic 
Levi-flat hypersurface. The criterion is stated in terms of the Segre varieties.
As an application, we obtain a structure theorem for some class of currents in 
the nondicritical case.
\end{abstract}

\maketitle

\let\thefootnote\relax

* Department of Mathematics, Indiana University, 831 E 3rd St. Rawles Hall, Bloomington,   IN 47405, USA, e-mail: pinchuk@indiana.edu

** Department of Mathematics, the University of Western Ontario, London, Ontario, N6A 5B7, Canada,
e-mail: shafikov@uwo.ca. The author is partially supported by the Natural Sciences and Engineering 
Research Council of Canada.

***Universit\'e  de Lille (Sciences et Technologies), 
U.F.R. de Math\'ematiques, 59655 Villeneuve d'Ascq, Cedex, France,
e-mail: sukhov@math.univ-lille1.fr. The author is partially supported by Labex CEMPI.

\section{Introduction}
The study of Levi-flat hypersurfaces  arises naturally in several areas of complex geometry. Our approach is inspired by the theory of holomorphic foliatons. This aspect of Levi-flat geometry was considered by several authors \cite{Bru1,Bru2,BuGo, CerNeto,CerSad,Fe,Le,FeLe,SS}. By the classical theorem of E.~Cartan, a nonsingular real analytic Levi-flat hypersurface is locally biholomorphic to a real hyperplane. The present paper studies local properties of Levi-flat hypersurfaces near singular points.

Our main result (Theorem \ref{MainTheo}) gives a complete effective characterization of dicritical singular points of a Levi-flat real analytic hypersurface in terms of the geometry of its Segre varieties. This answers the question communicated to the second and the third author by Jiri Lebl (see also \cite{FeLe}). As an application we prove a structure theorem for currents supported on nondicritical hypersurfaces (Proposition \ref{MainProp}).

This paper was written when the third author visited Indiana University (Bloomington) during the Spring semester of 2016. He expresses his gratitude for excellent work conditions.

\section{Real analytic Levi-flat hypersurfaces in $\C^n$}

\subsection{Real analytic sets and their complexification}
Let $\Omega\subset \mathbb R^n$ be a domain. A real analytic set $\Gamma \subset \Omega$ is a closed set locally defined 
as a zero locus of a finite collection of real analytic functions. In fact, we can always take just one function to locally define any 
real analytic set. We say that $\Gamma$ is {\it irreducible} in $\Omega$ if it cannot be represented as the union $\Gamma = \Gamma_1 \cup \Gamma_2$ of two real analytic sets $\Gamma_j$ in $\Omega$ with $\G_j \setminus (\G_1 \cap \G_2) \ne \varnothing$,
$j=1,2$, (this is the geometric irreducibility). 
$\Gamma$ is called a {\it real hypersurface} 
if there exists a point $q \in \G$ such that near $q$ the set $\G$ is a real analytic submanifold of dimension $n-1$. 
For a real hypersurface $\Gamma$ we call such $q$ a {\it regular} point. The union of all regular points form a regular locus denoted by $\G^*$. Its complement 
$\Gamma_{sing}:= \Gamma \setminus \Gamma^*$ is called the {\it singular locus} of $\Gamma$. 
Note that our convention is different from the usual definition of a regular  point in semianalytic or subanalytic geometry where 
a similar notion is less restrictive and a real analytic set is allowed to be a submanifold of {\it some} dimension near a regular point. 
By our definition, points of a hypersurface $\Gamma$, where $\Gamma$ is a submanifold of dimension smaller than $n-1$, 
belong to the singular locus. For that reason $\Gamma^*$ may not be dense in $\Gamma$, this can happen even if 
$\Gamma$ is irreducible (so-called umbrellas). Note that $\G_{sing}$ is a closed 
semianalytic subset of $\Gamma$ (possibly empty) of real dimension at most $n-2$. 

In local questions we are interested in the geometry of a real hypersurface $\G$ in an arbitrarily small neighbourhood of a given point $a\in \G$, i.e., of the germ at $a$ of $\Gamma$. If the germ is irreducible at $a$ we may consider a sufficiently small open neighbourhood $U$ of $a$ and a representative of the germ which is irreducible at $a$, see \cite{Na} for details. In what follows we will not distinguish between the germ of $\Gamma$ at a given point $a$ and its particular representative in a suitable neighbourhood of~$a$.

Let $\Gamma\subset\mathbb R^n_x$ be the germ of a real analytic set at the origin. By $\Gamma^{\mathbb C}$ we denote 
the complexification of $\Gamma$, i.e., a complex analytic germ at the origin in 
$\cx^n_{z} = \mathbb R^n_x +i\mathbb R^n_y$, $z=x+iy$, with the property that any holomorphic function that vanishes 
on $\Gamma$ necessarily vanishes on $\Gamma^{\mathbb C}$. Equivalently, $\G^{\mathbb C}$ is the smallest complex analytic germ in $\mathbb C^n$ that contains $\Gamma$. It is well known that
the dimension of $\G$ equals the complex dimension of $\Gamma^{\mathbb C}$ and that the germ of $\Gamma^{\mathbb C}$ is 
irreducible at zero whenever the germ of  $\G$ is irreducible, see Narasimhan~\cite{Na} for further details and proofs. Also, given a real analytic germ $\sum_{|j|\ge 0} a_j\, x^j$, $a_j \in \mathbb R$, $x\in \mathbb R^n$, we define its 
complexification to be the complex analytic germ  $\sum a_j\, z^j$.

While the complexification of the germ of a real analytic set is canonical and is independent of the choice of the defining function, 
the next lemma gives a convenient way of constructing the complexification of a real analytic hypersurface using a suitably chosen defining function. We will need the following notion of a minimal defining function for a complex hypersurface. Given
a complex hypersurface $A = \{ z \in \Omega: f(z) = 0 \}$ in a domain $\Omega\subset \cx^n$, $f$  is called {\it minimal} if for every open subset $U \subset \Omega$ and any function $g$ holomorphic on $U$ and such that $g = 0$ on $A \cap U$, there exists a function $h$ holomorphic in $U$ such that $g = h f$. If $f$ is a minimal defining function, then the singular locus of $A$ coincides with the set $f = df = 0$. Locally, any
irreducible complex hypersurface admits a minimal defining function, see Chirka~\cite{Ch}. 

\begin{lemma}\label{l.df}
Let $\Gamma\subset\mathbb R^n$ be an irreducible  germ  of a real analytic hypersurface at the origin. Then there exists a defining function 
$\rho(x)$ of the germ of $\G$ at the origin such that its complexification $\hat \rho (z)$ is a minimal defining function of the 
complexification $\G^{\mathbb C}$.
\end{lemma}

\begin{proof}
Since the germ of $\Gamma$ is irreducible, the  complexification $\G^{\mathbb C}$ is  an irreducible germ of  complex hypersurface in $\cx^n$. It admits a minimal defining function at the origin
$F(z)=\sum_{|j|>0} c_j\, z^j$. Let $c_j = a_j + i b_j$, $a_j, b_j \in \mathbb R$. Let $\hat f(z) = \sum a_j\, z^j$,
$\hat g(z) = \sum b_j\, z^j$, so that $F = \hat f + i\hat g$. Then $\hat f$ and $\hat g$ are the complexifications of 
real analytic germs $f(x) = \sum a_j x^j$ and $g(x) = \sum b_j x_j$ respectively. Moreover, since 
$F (z) \left|_{\mathbb R^n_x} = f + i g \right. $, and $F(x)$ vanishes on $\Gamma$, we conclude that both $f$ and $g$
vanish on $\G$, and therefore, $\hat f$ and $\hat g$ vanish on $\G^{\mathbb C}$. Since $F$ is the minimal defining function
for $\G^{\mathbb C}$, there exist unique holomorphic germs $h_1$ and $h_2$ such that $\hat f = h_1 F$, $\hat g = h_2 F$. But then $F = (h_1  + i h_2) F$, i.e., $h_1 + i h_2 =1$ identically. Hence, at least one of these functions, say $h_1$, does not vanish at the origin. It follows that $F = {h_1}^{-1}\hat f$, i.e., $\hat f$
is also a minimal defining function of $\G^{\mathbb C}$. Thus, $\rho = f$ is the required choice of a defining function of $\Gamma$. \end{proof}

\subsection{Levi-flat hypersurfaces}
Let $z=(z_1, \dots, z_n)$, $z_j = x_j + i y_j$, be the standard coordinates in~$\mathbb C^n$.
Let  $\Gamma$ be an irreducible germ of a real analytic hypersurface at the origin defined by a function $\rho$ provided by Lemma \ref{l.df}. In a (connected) sufficiently small neighbourhood of the origin   $\Omega \subset \C^n$ the hypersurface $\G$ is a closed irreducible real analytic subset 
of $\Omega$ of dimension $2n-1$. 

For $q\in \G^*$ consider the complex tangent space $H_q(\Gamma):= T_q(\Gamma) \cap JT_q(\Gamma)$.  The {\it Levi form} 
of $\Gamma$ is a Hermitian quadratic form defined on $H_q(\Gamma)$ by
$L_q(v) = \sum_{k,j} \rho_{ z_k  \overline{z}_j}(q) v_k \overline{v}_j$ with $v \in H_q(\Gamma)$.
A real analytic hypersurface $\Gamma$ is called {\it Levi-flat} if its Levi form vanishes on $H_q(\Gamma)$ for every regular point $q$ of $\Gamma$. 
By the classical result of Elie Cartan, for every point  $q \in \Gamma^*$ there exists a local biholomorphic change of coordinates centred at $q$ such 
that in the new coordinates $\Gamma$ in some neighbourhood $U$ of  $q = 0$ has the form $\{ z \in U: z_n + \overline{z}_n = 0 \}$ . Hence, $\Gamma \cap U$ is locally foliated by complex hyperplanes $\{z_n = c, \ c \in i\,\mathbb R\}$. 
This foliation is called {\it the Levi 
foliation} of $\Gamma^*$, and  will be denoted  by~$\mathcal L$. 
We denote by ${\mathcal L}_q$ the leaf of the Levi foliation through $q$.
Note that by definition it is a connected complex hypersurface closed in $\Gamma^*$.

Let $0\in \overline \Gamma^*$. We choose the neighbourhood $\Omega$ of the origin in the form of a polydisc  
$\Delta(\varepsilon) = \{ z \in \C^n: \vert z_j \vert < \varepsilon \}$
 of radius $ \varepsilon > 0$.  Then for $\varepsilon$ small enough, the function $\rho$ admits the  Taylor  expansion 
 convergent in $U$:
\begin{eqnarray}\label{exp}
\rho(z,\overline z) = \sum_{IJ} c_{IJ}z^I \overline{z}^J, \ c_{IJ}\in\cx, \ \ I,J \in \N^n.
\end{eqnarray}
The coefficients $c_{IJ}$ satisfy the condition
\begin{eqnarray}
\label{coef}
\overline c_{IJ} = c_{JI},
\end{eqnarray}
because $\rho$ is a real-valued function. Note that in local questions we may further shrink $\Omega$ as needed.

For real analytic sets in complex manifolds it is more convenient to define the complexification as follows.
Denote by $J$  the standard complex structure of $\cx^n_z$, and let $J'$ on $\C^n_w$ be defined as $J'w = -i w$. 
We equip $\C^{2n} = \C^n_z \times \C^n_w$ with the complex structure $J \otimes J'$. Then
the map $\iota : \cx^n \to \cx^n \times \cx^n$ given by $z\to (z, z)$ is a totally real embedding of $\cx^n$ into 
$(\cx^{2n}, J \otimes J')$. We define the complexification of a real analytic germ $\G\subset \mathbb C^n$ to be 
the smallest complex analytic germ in $\cx^{2n}$ that contains $\iota(\G)$. This is an equivalent construction to that 
defined in the previous subsection, and so all the properties of the standard complexification are preserved. 
Now, given a real analytic germ $\rho$ as in~\eqref{exp}, its complexification is defined as 
\begin{eqnarray}\label{compl}
\rho(z,\overline w) = \sum_{IJ} c_{IJ}z^I\overline{w}^J ,
\end{eqnarray}
i.e., we replace the variable $\overline z$ with an independent variable $\overline w$.
We assume that $\varepsilon > 0$  is chosen so small that the series (\ref{compl}) converges for all $(z,w) \in \Delta(\varepsilon) \times \Delta(\varepsilon)$. Note that $\rho(z,\overline w)$ is a holomorphic function in $(z,w)$ by the choice of the complex structure on $\cx^{2n}$. If the reader prefers to work with the standard structure on $\C^{2n}$, then $\overline w$ should be 
appropriately replaced with $w$.

By Lemma~\ref{l.df}, the choice of the defining function $\rho$ guarantees that the complexification of (the 
germ of) $\Gamma$ is given by
\begin{equation}\label{e.gammac}
\G^\C = \{ (z,w) \in \C^n \times \C^n: \rho(z,\overline w) = 0 \}.
\end{equation}
The hypersurface $\G$ lifts canonically to $\G^\C$ as $$\hat\G = \G^\C \cap \{  w = z \} .$$
In what follows we denote by $\G^\C_{sing}$ the singular locus of $\G^\C$. 

\subsection{Segre Varieties.} Our key tool is the family of  {\it Segre varieties} associated with a real 
analytic hypersurface~$\Gamma$.
For $w \in \Delta(\varepsilon)$ consider a complex analytic hypersurface given by
\begin{eqnarray}
\label{Segre1.1}
Q_w = \{ z \in \Delta(\varepsilon) : \rho(z,\overline w) = 0 \} .
\end{eqnarray}
It is called the {\it Segre variety} of the point $w$. This definition uses the defining function $\rho$ of $\Gamma$ in a neighbourhood of the origin which appears in (\ref{e.gammac}).  We will always consider the case where  the germ of $\Gamma$ at the origin is irreducible and  everywhere through the paper we  use a defining function provided by Lemma \ref{l.df} in a neighbourhood of the origin (the same convention is used in \cite{SS}). In general  the Segre varieties $Q_w$ also  depend on the choice of   $\varepsilon$ (some irreducible components of $Q_w$ may disappear when we shrink $\varepsilon$). Throughout the paper we consider only the Segre varieties $Q_w$ defined by means of the complexification at the origin. The reader should keep this in mind. Also note that if $0$ is a regular point of  $\Gamma$, then the notion of the Segre variety $Q_w$ is independent of the choice of a defining function  $\rho$ with non-vanishing gradient when $w$ is close enough to the origin.

The following properties of Segre varieties are immediate.

\begin{lemma}\label{SegreProp}
Let $\G$ be a germ of an irreducible real analytic hypersurface in $\mathbb C^n$, $n>1$. Then
\begin{itemize}
\item[(a)] $z \in Q_z$ if and only if  $z \in \Gamma$,
\item[(b)] $z \in Q_w $ if and only if $w \in Q_z$.
\end{itemize}
\end{lemma}
We also recall the property of local biholomorphic invariance of  some distinguished components of the Segre varieties near regular points. Since here we are working near a singularity we state this property in detail using the notation introduced above. Consider  a regular point $a \in \G^* \cap \Delta(\varepsilon)$ and fix $\alpha > 0$ small enough with respect to $\varepsilon$. Consider any  function $\rho_a$ real analytic on the polydisc $\Delta(a,\alpha)= \{ \vert z_j-a_j \vert < \alpha ,j=1,...,n\}$
such that $\Gamma \cap \Delta(a,\alpha) = \rho_a^{-1}(0)$ and the gradient of $\rho_a$ does not vanish on $\Delta(a,\alpha)$. Then for $w \in \Delta(a,\alpha)$ we can define the Segre variety ${}^aQ_w$ ("the Segre variety with respect to the regular point $a$") as 
$$
{}^aQ_w = \{ z \in \Delta(a,\alpha): \rho_a(z,\overline w) = 0 \} ,
$$ 
(we use the Taylor series of $\rho_a$ at $a$ to define the complexification). For $\alpha$ small enough, ${}^aQ_w$  is a connected  nonsingular complex submanifold of dimension $n-1$ in $\Delta(a,\alpha)$. This definition is independent of the choice of the local defining function $\rho_a$ satisfying the above properties. We have the inclusion  ${}^aQ_w \subset Q_w $. Note that in general  $Q_w$ can have  irreducible components in $\Delta(\varepsilon)$ which do not contain ${}^aQ_w$.

\begin{lemma}(Invariance property) Let $\Gamma$, $\Gamma'$ be irreducible germs of real analytic hypersurfaces, 
$a \in \Gamma^*$, $a' \in (\Gamma')^*$, and $\Delta(a,\alpha)$, $\Delta(a',\alpha')$ be small polydiscs. Let
$f: \Delta(a,\alpha) \to \Delta(a',\alpha')$ be a holomorphic map such that $f(\Gamma \cap \Delta(a,\alpha)) \subset \Gamma' \cap \Delta(a',\alpha')$ and $f(a) = a'$. Then
$$f({}^aQ_w) \subset {}^{a'}Q'_{f(w)}$$
for all $w \in \Delta(a,\alpha)$ close enough to $a$. In particular, if $f:\Delta(a,\alpha) \to \Delta(a',\alpha')$ is biholomorphic, then $f({}^aQ_w) = {}^{a'}Q'_{f(w)}$.
Here ${}^aQ_w$ and ${}^{a'}Q'_{f(w)}$ are the Segre varieties associated with $\Gamma$ and $\Gamma'$ and the points $a$ and $a'$ respectively.
\end{lemma}
For the proof see for instance, \cite{DiPi}.
As a simple consequence of Lemma \ref{SegreProp} we have 

\begin{cor}
\label{Segre+Levi}
Let $\Gamma\subset\C^n$ be an irreducible  germ  at the origin of a real analytic Levi-flat hypersurface. Let  $a \in \Gamma^*$. Then the following holds:
\begin{itemize}
\item[(a)]  There exists a unique irreducible component $S_a$ of $Q_a$ containing the leaf ${\mathcal L}_a$. This is also a unique 
complex hypersurface through $a$ which is contained in $\Gamma$.
\item[(b)] For every $a, b \in \Gamma^*$ one has $b \in S_a \Longleftrightarrow S_a = S_b$.
\item[(c)] Suppose that $a \in \Gamma^*$ and ${\mathcal L}_a$ touches a point  $q \in \Gamma$ such that $\dim_\C Q_q = n-1$ (the point $q$  may be singular). Then  $Q_q$ contains $S_a$ as an irreducible component. 
\end{itemize}
\end{cor}
The proof is contained in \cite{SS}. Again, we emphasize that Corollary~\ref{Segre+Levi} concerns the ``global" 
Segre varieties, i.e., those defined by (\ref{Segre1.1}) using the complexification at the origin.

\section{Characterization of dicritical singularities for Levi-flat hypersurfaces}

Let $\G$ be an irreducible  germ of  a  real analytic Levi-flat hypersurface in $\C^n$  at  $0 \in \overline{\G^*}$. 
Fix a local defining function $\rho$  
chosen by Lemma \ref{l.df} so that the complexification $\G^\C$ is an  irreducible germ of a complex hypersurface in  $\C^{2n}$
given as the zero locus of the complexification of $\rho$. As already mentioned above, all Segre varieties which we consider are  defined by means of this complexification at the origin.

Fix also $\varepsilon > 0$ small enough; all considerations are in the polydisc $\Delta(\varepsilon)$ centred at the origin. A point $q  \in \overline{\G^*} \cap \Delta(\varepsilon)$ is called a {\it dicritical} singularity if $q$ belongs to the closure of infinitely many geometrically different leaves ${\mathcal L}_a$. 
Singular points in $\overline{\G^*}$ which are not dicritical are called {\it nondicritical}. 

A singular point $q$ is called {\it Segre degenerate} if $\dim Q_q = n$.  We recall  that the  Segre degenerate singular points form a complex analytic subset 
of $\Delta(\varepsilon)$ of complex dimension at most $n-2$,  in particular, it is a discrete set if $n=2$.  For the proof see \cite{Le,SS}.
The main result of this paper is the following

\begin{theorem}\label{MainTheo}
Let $\G = \rho^{-1}(0)$ be  an irreducible germ at the origin of a  real analytic Levi-flat hypersurface in $\C^n$  and  $0 \in \overline{\G^*}$. Then $0$  is a dicritical point if and only if it is Segre degenerate.
\end{theorem}

\begin{proof}   A dicritical point is Segre degenerate; this  follows from Corollary ~\ref{Segre+Levi}(c). 
We prove now that if the origin is a  Segre degenerate point then it  is dicritical. The proof is divided into four steps.

\smallskip

{\bf Step 1. Canonical Segre varieties.} 
Consider the canonical projection $$\pi: \G^\C \to \C^n, \ \ \pi: (z,w) \mapsto w.$$ Then 
$Q_w = \pi^{-1}(w) \,\,\,\mbox{for every} \,\,\, w $.
Denote by $Q_w^c$ the union of irreducible components of $Q_w$ containing the origin; we call this {\it the canonical Segre variety}
of $w$. Note that for all $w$ from a neighbourhood of the origin in $\C^n$ its canonical Segre variety $Q_w^c$ is a 
nonempty complex analytic hypersurface.  Indeed, since $0$ is a Segre degenerate singularity, $w \in Q_0 = \C^n$, and 
by Lemma~\ref{SegreProp}(b) we obtain $0 \in Q_w$. 

Consider the set 
$$\Sigma = \{ (z,w) \in \G^\C: z  \notin Q_{w}^c \} .$$

If  $\Sigma$ is empty, then for every point $w$ from a neighbourhood of the origin the Segre variety $Q_{w}$ coincides 
with the canonical Segre variety $Q_w^c$, i.e., all components of $Q_w$ contain the origin. But for  a regular point $w$ 
of $\G$, the closure of its Levi leaf is  a component of $Q_{w}$. Therefore, the closure of every Levi leaf contains the origin 
which is then necessarily a dicritical point. Our goal is to prove that $\Sigma$ is the empty set. Arguing by contradiction assume that $\Sigma$ is not empty.  Observe  that the set $\Sigma$ is open in $\G^\C$. This follows immediately from the fact 
that the defining function of a complex hypersurface  $Q_{w}$ depends continuously on the parameter $w$. 

\smallskip

To prove the theorem, we are going to show that the boundary of $\Sigma$ is contained in a proper complex analytic subset 
of $\G^\C$. For this we introduce the following set. Let 
$$X = \{ (z,w) \in \G^\C: \dim Q_w = n \} .$$
As shown in \cite{Le,SS} the set $X$ is contained in a complex analytic subset of $\G^\C$ of dimension at most $2n-2$. 

Let $(z^k,w^k)$ be a sequence of points from $\Sigma$ converging to 
some $(z^0,w^0)$. Without loss of generality we may assume that $(z^0,w^0) \in \G^\C \setminus (\Gamma^\C_{sing} \cup X )$, and that $(z^0,w^0)$ does not belong to~$\Sigma$. Since $(z^0,w^0)$ does not belong to $X$, we conclude that $w^0$ is not a 
dicritical singularity. Then $Q_{{w^0}}$ is a complex hypersurface (in general, reducible) passing through the origin, and 
$z^0 \in Q_{w_0}^c$. 

\smallskip

{\bf Step 2. Analytic representation of  the Segre varieties.}   
We use the notation $z = (z',z_n) = (z_1,...,z_{n-1},z_n)$. 
Performing a complex linear change of coordinates in $\C^n_z$ if necessary,
we can assume that the intersection of $Q_{{w^0}}$ with the $z_n$-coordinate complex line $ (0',\C)$ is a discrete set. 
Then, the intersection of  $\Gamma^{\C}$ with the complex line $\{(0',\C,w_0)\}$ is also discrete. 
Let 
\begin{equation*}
 \tilde \pi (z', z_n,w) = (z', w)
\end{equation*}
be the coordinate projection. Choose a neighbourhood $U$ of the origin in $\C^n_z$ and a neighbourhood $V$ of $w^0$ in $\C^n_w$ 
with the following properties:
\begin{itemize}
\item[(i)] $U = U' \times \delta\D$, where $U'$ is a neighbourhood of the origin in $\C^{n-1}_{z'}$, and $\D$ is
the unit disc in $\mathbb C$. Choose $\delta>0$ small so that 
$$
\{|z_n| < \delta \} \cap \Gamma^\C \cap \tilde\pi^{-1}(0',w_0) = \{(0,w_0)\} .
$$
\item[(ii)]   The projection $\tilde\pi : \Gamma^{\C}\cap (U\times V) \mapsto U' \times V$ is proper.
\end{itemize}
We apply the Weierstrass preparation theorem to the equation~\eqref{e.gammac} on the neighbourhood $U \times V$ of the 
point $(0, w_0)\in \G^{\cx}$ to obtain
\begin{eqnarray}\label{gamma}
\G^\C  = \{ (z,w)\in U\times V : P(z',{\overline w})(z_n):=z_n^d + a_{d-1}(z', {\overline w})z_n^{d-1} + ...
+ a_0(z',{\overline  w}) = 0 \},
\end{eqnarray}
where the coefficients $a_j(z',{\overline w})$ are holomorphic in $(U' \times V)$. Note that $a_0 (0',{\overline w}) = 0$ for
all $w$ because every Segre variety contains the origin. 
The Segre varieties then are obtained by fixing $w$ in the above equation:
\begin{eqnarray}\label{Q}
Q_{w} \cap U= \{ z \in U : P(z', {\overline w})(z_n)= 0 \}, \ \ w \in V .
\end{eqnarray}

\smallskip

{\bf Step 3: Boundary points of $\Sigma$.}   As we have noted in Step 1,  the set $\Sigma$ is open in $\G^\C$.  
In this step we show that in a neighbourhood of $(z^0,w^0)$ the boundary of $\Sigma$ is contained in a proper complex 
analytic subset of $\G^\C$. 

We will need an analytic representation of $\Gamma^\C$ similar to (\ref{gamma}) but in a neighbourhood of the point $(z^0,w^0)$.  Performing a (arbitrarily close to the identity map) linear change of coordinates in $\C^n_z$, we can assume that Step 2 holds and, additionally, the intersection of $Q_{w^0}$ with the complex line $(z^0_1,...,z^0_{n-1},\C)$ is discrete. As in Step 2,  
there exist a neighbourhood $O'$ of $(z^0_1,...,z^0_{n-1})$ in $\C^{n-1}$  and $\delta' > 0$ such that  $\G^\C  \cap (O \times V)$ is defined as the zero set of some Weierstrass polynomial $\tilde P(z',\overline{w})(z_n-z^0_n)$. Here $O = O' \times \delta'\D$ and $V$ is the same neighbourhood of $w^0$ as in Step 2 (this can be achieved by shrinking $V$ if necessary); the polynomial $\tilde P$ has the expansion similar to (\ref{gamma}) with $(z_n - z_n^0)$ instead of $z_n$ and its coefficients are holomorphic on $O' \times V$. For the Segre varieties $Q_w$, $w\in V$, we have
\begin{eqnarray}\label{Segre1}
Q_w  \cap O  = \{ z \in O  : \tilde P(z',{\overline w})(z_n - z^0_n)= 0\} .
\end{eqnarray}

Consider now the discriminant $R(z',w)$ of the polynomial $\tilde P$ that is the resultant of $\tilde P$ and its derivative in $z_n$, 
see, for example, \cite{Ch}. The function $R$ is holomorphic in $O' \times V$ and we define the discriminant set as
\begin{eqnarray}
\label{gamma4}
Y = \{ (z,w) \in \Gamma^\C \cap (O \times V) : R(z',\overline w) = 0 \} .
\end{eqnarray}
The projection of the set $Y$ on $\C^{n-1}_{z'} \times \C^n_w$ is formed by the points $(z',\overline{w})$ such that the polynomial $\tilde P(z',\overline w)$ has multiple roots. The set $Y$ is a  complex analytic subset of codimension 1 in $\Gamma^\C  \cap (O \times V)$.    We have the inclusion $\G^\C_{sing} \cap (O \times V) \subset Y$. In general, this inclusion is strict, see, e.g., \cite{Ch}.

Now we use again a neighbourhood $U$ of the origin in $\C^n_z$ and a neighbourhood $V$ of $w^0$ defined in  Step 2,
so that conditions (i), (ii) of Step 2 are satisfied. In particular, $Q_w \cap U$ is given by (\ref{Q}) for each $w \in V$. Set $z' = 0$ in  the equation (\ref{Q}).  This defines an algebroid $d$-valued function in $w \in V$, that is, an algebraic element over the commutative integral domain of functions holomorphic on~$V$. More precisely, consider $(\zeta,w) \in \C \times V$ satisfying the equation 
 \begin{eqnarray}\label{algebr}
 \zeta^d + a_{d-1}(0', {\overline w})  \zeta^{d-1} + ...+ a_0(0', {\overline w}) = 0 , 
\end{eqnarray}
where  $a_j$  are the coefficients of the polynomial $P$ from (\ref{gamma}). This equation defines an algebroid ($d$-valued) function 
$ w\mapsto\zeta(w)$ (in other words, $\zeta$ is a holomorphic correspondence defined on $V$ and with values in $\C$). The complex hypersurface determined by the equation (\ref{algebr}) in $\C \times V$ is a branched analytic covering over $V$, and we can in a standard way define the branches of the algebroid function $\zeta$ as holomorphic functions over a simply connected domain in $V$ which does not intersect the branch locus,  see \cite{Ch}. 
Furthermore, given $w \in V$ the algebroid function $\zeta$ associates the set $\zeta(w) = (\zeta_1(w),...,\zeta_s(w))$, 
$s = s(w) \le d$,  of the (distinct) roots of the equation (\ref{algebr}); we refer to them as the values of $\zeta$ at $w$.
Since $a_0(0',  \overline{w})$ vanishes identically in $w$ (recall that every Segre variety $Q_w$ contains the origin), one of the branches of $ \zeta$ is equal to zero identically; in particular,  the polynomial (\ref{algebr}) is reducible. On the other hand, the function $ \zeta$ has branches which are not equal to zero identically. Indeed, $(z^k,w^k)  \in  \Sigma$ so that the irreducible components of $Q_{w^k}$ containing $z^k$ do not contain the origin. Therefore, the equation (\ref{algebr}) has non-zero solutions  when $w = w^k$;  in particular, $a_i(0',w^k)  \neq 0$ for at least one $i$. Let $j$ be the smallest index such that the coefficient  $a_j(0',\overline w)$ does not vanish identically. Dividing 
equation~\eqref{algebr} by $\zeta^{j}$ we obtain
\begin{equation}\label{e.agl}
\zeta^{d-j} + a_{d-1}(0', \overline w)\, \zeta^{d-j-1} + ...+ a_j(0', \overline w) = 0 .
\end{equation} 
Thus, all non-zero values  of the algebroid function $\zeta$ at $w$  are solutions of this equation. 

Note that $0$ is one of the roots of the equation (\ref{e.agl}) for some $w$   if and only if 
$a_j(0',\overline{w}) = 0$. Define the set 
\begin{eqnarray}
\label{A}
A = \{ (z,w) \in \G^\C: a_j(0',\overline{w}) = 0 \}. 
\end{eqnarray}
This is a complex analytic subset of codimension 1 in $\G^\C$.

 \begin{lemma}
 The boundary of $ \Sigma$ in a neighbourhood of $(z^0,w^0)$ is contained in the union $A  \cup X  \cup Y$.
  \end{lemma}
  
\begin{proof} It suffices to consider the case when the point $(z^0,w^0)$ does not belong to $X  \cup Y$. We use neighbourhoods $O \ni z^0$ and $V \ni w^0$ defined at the beginning of Step 3; we also use the representation (\ref{Segre1}) for $Q_w \cap O$ with $w \in V$.

 Since the point  $(z^0,w^0)$ is not in $Y$,  the polynomial  $\tilde P((z^0)',\overline{w}^0)(z_n-z^0_n)$    in (\ref{Segre1})  does not have multiple roots.  It follows that this point  is  regular   for $\G^\C$ and that the point $z^0$ is regular for the Segre variety $Q_{{w^0}} $. The points  $(z^k,w^k)$ also do not belong to $Y$ for $k$ big enough and  are regular  points for  $\G^\C$ and for $Q_{w^k}$.

Let $K_1(w),...,K_m(w)$  be the irreducible components of $Q_w$, $w \in  V$. The point $(z^0,w^0)$
  belongs to exactly one of these components, say, to $K_1(w^0)$. Since $Q_{w^0}$ has the maximal number  of branches over  the point $(z^0)'$, any two distinct  components $K_\nu(w^k)$, $\nu=1,...,m$, of $Q_{w^k}$ cannot  glue together when $w^k$ tends to $w^0$. Therefore, $K_1(w^0) \cap O$ is contained as an irreducible component of the limit set (in the Hausdorff distance) of exactly one of these components,   as $w^k \to w^0$. By the uniqueness theorem for irreducible complex analytic sets, this property  holds not only on $O$, but also globally (in particular, in a neighbourhood of the origin in $\C^n$).  Denote this component by $K_1(w^k)$; note  that $K_1(w^k)$ is a unique component containing $z^k$ for $k$ big enough. 

 It follows from the representations (\ref{gamma}) and (\ref{Q}) that   for every $w = w^k$ or $w = w^0$ the fibre $\tilde \pi^{-1}(0',w) \cap K_1(w)$ is a finite set, which we write in the form $\{p^1(w), \dots, p^l(w) \}$ , $l = l(k) \le d$. Since $K_1(w^k)$ is a component of $Q_{w^k}$, each $p^\mu_n(w^k)$ is a value of the algebroid function $\zeta$ at $w^k$, i.e., belongs to the set $\zeta(w^k)$.  Recall that  $(z^k,w^k) \in \Sigma$, and the component $K_1(w^k)$ does not contain the origin. This 
implies that $p^\mu_n(w^k)  \neq 0$ for all $\mu=1,...,l$. Hence all values $p^\mu_n(w^k)$ satisfy the equation (\ref{e.agl}) with $w = w^k$. By the choice of $K_1(w^k)$,
the set $(p^1_n(w^0),...,p^l_n(w^0)$ is contained in the limit set of the sequence $(p^1_n(w^k),...,p^l_n(w^k))$ as $w^k \to w^0$. Therefore, every 
$p^\mu_n(w^0)$ satisfies the equation (\ref{e.agl}) with $w = w^0$. But the point $(z^0,w^0)$ does not belong to $\Sigma$ and the component $K_1(w^0)$ necessarily contains the origin. This means that $p_n^\mu(w^0) = 0$  for at least one index $\mu$. We obtain that    $a_j(0', \overline{w}^0) = 0$ and $(z^0,w^0)  \in A$.
 \end{proof}

Now by the Remmert-Stein removable singularity theorem, the closure 
$\overline\Sigma$ of $\Sigma$  coincides with an irreducible component of  $\G^\C$. Since the complexification $\G^\C$ is irreducible, we obtain that the closure $\overline\Sigma$ of $\Sigma$ coincides  with all of $\G^\C$.

\smallskip

{\bf Step 4: The complement of $\Sigma$ has nonempty interior.} 
We begin with the choice of a suitable point $\hat w$. First assume that $(\hat z, \hat w)$ is a regular point of $\G^\C$ and $(\hat z, \hat w)$ is not in $X$. Fix a neighbourhood $W$ of $\hat w$ small enough. Then for all Segre varieties $Q_w$, $w \in W$ the number of their irreducible components is bounded above uniformly in $w$.
Let $m$ be the maximal number of components of $Q_{ w}$ for $w\in W$. Slightly perturbing $\hat w$ (and $\hat z$), one can assume that  $\hat w$ is such that $Q_{\hat w}$ has exactly $m$ geometrically distinct components. Then there exists a neighbourhood $V$ of $\hat w$
such that $Q_w$ has exactly $m$ components  for all $w\in V$. Let $K_1(\hat w), \dots, K_m(\hat w)$ be the irreducible components 
of $Q_{\hat w}$. Note that the components $K_j(w)$ depend continuously on $w$ in $V$.

Consider the sets $F_j = \{ w \in V: 0 \in K_j(w) \}$. Every set $F_j$ is closed in $V$. Since $0\in Q_w$ for every $w$, 
$\cup_j F_j = V$. Therefore, one of these sets, say, $F_1$, has a nonempty interior. This means that there exists a small ball $B$ 
centred at some $\tilde w$ such that $K_1(w)$ contains $0$ for all $w \in B$. Choose a regular point $\tilde z$ in $K_1(\tilde w)$ 
close to the origin. Then for every $(z,w) \in \Gamma^\C$ near $(\tilde z, \tilde w)$ we have $z \in K_1(w)$, i.e., $(z,w)\notin \Sigma$. Hence, the 
complement of $\Sigma$ has a nonempty interior. But this contradicts the conclusion of Step~3 that $\overline \Sigma = 
\G^{\mathbb C}$, and the proof is complete. \end{proof}

\section{Uniformly laminar currents near nondicritical singularities}

 We say that the Segre variety $Q_w$ defined by (\ref{Segre1}) is {\it minimal} if the holomorphic function 
 $z \mapsto \rho(z,\overline{w})$ is minimal. We have  the following

\begin{prop}\label{trans}
Let $\Gamma$ be a  real analytic Levi-flat hypersurface in $\C^{n}$ with an irreducible germ at the origin. Assume that $0$ is a nondicritical singularity for $\Gamma$. For a sufficiently small neighbourhood $\Omega$ of the origin there exists a complex linear map  $L: \C \to  \C^{n}$ with the following properties:
\begin{enumerate}
\item[(i)] $L(\cx) \cap Q_0  = \{0\}$.
\item[(ii)] No component of the 1-dimensional real analytic set $\gamma =L(\cx) \cap \Gamma$ is contained in  $\Gamma_{sing}$. 
\item[(iii)] For every $q\in \Gamma^*\cap \Omega$, there exists a point $w\in \gamma$ 
such that $\mathcal L_q$ is contained in  $Q_w$. 
\item[(iv)] If additionally the Segre variety $Q_0$ is  irreducible and minimal, then such a point  $w$ is unique.
\end{enumerate}
\end{prop}

Parts (i), (ii), and (iii)  are  proved in \cite{SS} (Proposition 4.1) under the assumption that $0$ is a Segre nondegenerate singularity. 
Theorem \ref{MainTheo} allows us to apply this result to the nondicritical case. Note that if $Q_0$ is irreducible and minimal, then Segre varieties $Q_w$ with $w$ close enough to the origin enjoy the same properties. This implies (iv).

A  1-dimensional real analytic set  $\gamma$ constructed in Proposition \ref{trans} is called {\it a transverse} for the Levi-flat 
$\Gamma$ at a nondicritical singularity. In general, $\gamma$ can be reducible, i.e., be a finite union of real analytic curves. 
The existence of a transverse shows that the structure of a Levi-flat hypersurface near a nondicritical singularity is similar to that of a nonsingular foliation. In \cite{SS} Proposition~\ref{trans} was used in order to extend a nondicritical  Levi foliation as a holomorphic web in a full neighbourhood of a singularity in $\C^n$. Here we give another application.

We use the standard terminology and notation  from the theory  of  currents, see \cite{Ch,D}. Denote by 
$ {\mathcal D}'_{p,q}( \Omega)$ 
the space of currents of bidimension $(p,q)$ (or simply  $(p,q)$-currents) in a domain  $\Omega$ of~$\C^n$.   
If  $A$ is a complex analytic subset of $\Omega$ of pure dimension $p$, then  $[A] \in {\mathcal D}'_{p,p}(\Omega)$ 
denotes the current of integration over $A$.

The main result of this section is 

\begin{prop}
\label{MainProp}
Let $\Gamma = \rho^{-1}(0)$ be a real analytic Levi-flat hypersurface   in $\C^{n}$ with the irreducible germ at the origin. Suppose that $0$ is a nondicritical singularity. Let also a 1-dimensional real analytic subset $\gamma$ in $\Gamma$ be a transverse containing the origin. Assume that the  Segre variety $Q_0$ is  irreducible and minimal. Furthermore,  suppose that $Q_s \setminus \Gamma_{sing}$, $s \in \gamma$,  are connected. 

Then there exists a neighbourhood $\Omega$ of the origin in $\C^n$  with the following property: every closed positive current $T \in {\mathcal D}'_{n-1,n-1}(\Omega)$ of order (of singularity) $0$ with support in 
$\overline{\Gamma^*}$ can be written in the form
\begin{eqnarray}
\label{laminar1}
T = \int_{s \in \gamma} [Q_s] d\mu(s)
\end{eqnarray}
with a unique positive measure $\mu$. 
\end{prop}
In the smooth case (for $C^1$ Levi-flat CR manifolds without singularities) this result is due to Demailly \cite{D1}. 
Proposition~\ref{MainProp} shows that every current $T$ satisfying the assumptions of the theorem is a so-called uniformly laminar current. These currents play important role in dynamical systems and foliation theory, see \cite{DG,FS}. Note that in many cases compact Levi-flat hypersurfaces in complex manifolds necessarily have singular points. This is our motivation for Proposition \ref{MainProp}.

We need some known  result on currents which we recall for the convenience of the reader. The proofs are contained 
in~\cite{D}. Recall that a current is called normal if both $T$ and $dT$ are currents of order zero.

\begin{prop}
\label{Sup1}
(First theorem of support) Let $T \in {\mathcal D}'_{p,p}(\Omega)$ be a normal current in a domain $\Omega$ in $\C^n$. If the support of $T$  is contained in a real manifold $M$ of CR dimension $< p$, then $T = 0$.
\end{prop}

Let $M$ be a Levi flat smooth hypersurface in $\Omega$ and $I$ be an (open) smooth real curve. Assume that there exists a submersion $\sigma: M \to I$ such that ${\mathcal L}_t = \sigma^{-1}(t)$ is a connected complex hypersurface (a Levi leaf) in $M$ for every $t \in I$. Our second tool is 

\begin{prop}
\label{Sup2}
(Second theorem of support) Any closed current $T \in {\mathcal D}'_{n-1,n-1}(\Omega)$ of order zero 
with support contained in $M$ can be written in the form 
$$T = \int_I [{\mathcal L}_t] d\mu(t)$$
with a unique complex measure $\mu$ on $I$. Moreover, $T$ is positive if and only if $\mu$ is positive.
\end{prop}

Let $A$ be an irreducible  complex $p$-dimensional analytic set in $\Omega$ and $T$ be a closed positive current of 
bidimension $(p,p)$ in $\Omega$. The generic Lelong number of $T$ along $A$ is defined as
$$m(A):= \inf\{ \nu(T,a) \ \vert \ a \in A \} .$$
Here $\nu(T,a)$ denotes the Lelong number of $T$ at $a$, which is defined as 
$$
\nu (T,a) = \lim_{r\to 0^+} r^{-2p}\int_{|z-a|<r} T \wedge \left(\frac{1}{2} dd^c |z|^2 \right)^p .
$$

 We need the following preparation result for Siu's semicontinuity theorem. Denote by ${\mathbf 1}_A$ the characteristic function of a set $A$.

\begin{prop}
\label{Sup3}
Let $T$ be a closed positive current of bidimension $(p,p)$ in $\Omega$ and let $A$ be an irreducible $p$-dimensional analytic subset of $\Omega$. Then ${\mathbf 1}_A T = m(A)[A]$.
\end{prop}

\begin{proof}[Proof of Proposition \ref{MainProp}]  This is a simple consequence of the existence of a transverse $\gamma$ given by 
Proposition~\ref{trans} and the above mentioned properties of currents. 

Since $Q_0$ is an irreducible  hypersurface with a minimal defining function, every $Q_{s}$, $s \in \gamma$, is an irreducible complex hypersurface for $s$ close enough to $0$ and is contained in $\Gamma$. The set of regular points of every $Q_s$ is connected. If a regular point of $Q_s$ belongs to $\Gamma^*$, then $Q_s$ coincides with some leaf of the Levi foliation near this point. However, a regular point of $Q_s$ in general can be a singular point 
of $\Gamma$. For this reason we impose the condition that  $Q_s \setminus \Gamma_{sing}$ 
are connected.

Consider the set $\gamma_{0} \subset \gamma$ which is defined as follows. First, it contains the singular points of the set $\gamma$. This is a finite set since $\gamma$ is real analytic. Furthermore, we include in $\gamma_{0}$   the points which are singular for $\Gamma$. Since $\gamma$ is not contained in $\Gamma_{sing}$, this is again a finite set. Furthermore, 
$\gamma_0$ contains the points $s$ such that the Segre variety $Q_s$ is contained in $\Gamma_{sing}$. Note that $\gamma_0$ is not empty since it contains $0$. Recall that $\Gamma_{sing}$ is a semianalytic set of dimension at most $2n-2$ and can be stratified into a finite union of  real analytic manifolds. In particular, it contains only a finite number of Segre varieties. Considering a small enough  neighbourhood $\Omega$ of the origin, we can assume that $\gamma_{0} = \{ 0 \}$. This is the reason why in the following argument 
we treat $Q_0$ in a special way; we do not assume, however, that $Q_0$ is  contained in $\Gamma_{sing}$.

Denote by $I$ one of the components of $\gamma \setminus \{ 0 \}$. Consider the domains $\Omega' = \Omega \setminus Q_{0}$ and  $\Omega'' = \Omega' \setminus \Gamma_{sing}$. The subset 
$$X = (\cup_{s \in I}Q_s) \setminus \Gamma_{sing}$$
is a closed smooth (without singularities) Levi-flat real analytic hypersurface in $\Omega''$. Furthermore, $X$ coincides with a component of $\Gamma^* \cap \Omega'$.

The positive current ${\mathbf 1}_X T$ is closed in $\Omega''$. By Proposition \ref{Sup2} we conclude that 
\begin{eqnarray}
\label{laminar2}
{\mathbf 1}_X T = \int_I [Q_s] d\mu(s)
\end{eqnarray}
for a unique positive measure $\mu$ on $I$.  Recall that $\dim  \Gamma_{sing} \le 2n-2$. By the choice of the neighbourhood of the origin, the only complex hypersurface that
$\G_{sing}$ may contain is $Q_0$, therefore, $\Gamma_{sing} \cap \Omega'$
can be stratified into a finite union of smooth real analytic CR manifolds of CR dimension $ < n-1$.
The current 
$$T\vert_{\Omega'} - \int_I [Q_s] d\mu(s)$$ 
is closed in $\Omega'$, is of order $0$,  and its support  is contained in $\Gamma_{sing}$. By Proposition \ref{Sup1} this current must vanish.
Hence, (\ref{laminar2}) holds on $\Omega'$. Repeating this argument for  other components of $\gamma \setminus \{0\}$,  
we extend $\mu$ on  $\gamma \setminus \{0\}$.

In order to extend $\mu$ to the origin we use Proposition~\ref{Sup3} which yields
$${\mathbf 1}_{Q_0} T = m(Q_{0}) [Q_{0}].$$
We set $\mu(0) = m(Q_{0})$. With this, $\mu$ is defined on $\gamma$ and (\ref{laminar1}) holds.
This completes the proof.
\end{proof}

The Segre varieties $Q_s$ are defined quite explicitly as the zero sets 
of the function $z \mapsto \rho(z,\overline s)$. In combination with the Poincar\'e-Lelong formula \cite{Ch,D}  this gives the following 
\begin{cor}
Under the assumptions of Proposition \ref{MainProp} we have 
\begin{eqnarray}
\label{PL}
T = \frac{i}{\pi}\int_{s \in \gamma} \partial\overline{\partial} \log \vert \rho(z,\overline{s})\vert  d\mu(s) .
\end{eqnarray}
\end{cor}
One can view (\ref{PL}) as the ``foliated" Poincar\'e-Lelong formula for nondicritical singularities.
Hence, nondicritical singularities are not "detected" on the level of currents: the structure is the same as in the smooth case. Only dicritical singularities are essential from this point of view.



\begin{thebibliography}{12}


\bibitem{Be} Bedford, E.: {\it Holomorphic continuation of smooth functions over Levi-flat hypersurfaces.} 
Trans. Amer. Math. Soc. 232 (1977), 323-341.

\bibitem{Bru1} Brunella M.: {\it Singular Levi-flat hypersurfaces and codimension one foliations.} 
Ann. Sc. Norm. Sup. Pisa {\bf VI} (2007), 661-672.

\bibitem{Bru2} Brunella M.: {\it Some remarks on meromorphic first integrals}, Enseign. Math. {\bf 58} (2012), 315-324.

\bibitem{BuGo} Burns, D., Gong, X.: {\it Singular Levi-flat real analytic hypersurfaces} Am. J. Math. {\bf 121} (1999), 23-53.

\bibitem{CerNeto} Cerveau D., Lins Neto, A.: {\it Local Levi-flat hypersurfaces invariants by a codimension one foliation.} 
Am. J. Math. {\bf 133} (2011), 677-716.

\bibitem{CerSad} Cerveau D., Sad, P.: {\it Fonctions et feuilletages Levi-flat. Etude locale.} 
Ann. Sc. Norm. Sup. Pisa {\bf III} (2004),427-445.

\bibitem{Ch} Chirka, E.: Complex analytic sets. Kluwer, 1989.

\bibitem{D} Demailly, J.-P. : Complex analytic and algebraic geometry. Book available online at 
http://www-fourier.uj-grenoble.fr/demailly/books.html.

\bibitem{D1} Demailly, J.P. {\it Courants positifs extremaux et conjecture de Hodge}, Invent. Math. {\bf 69} (1982), 347-374.

\bibitem{DiPi} Diederich, K., Pinchuk S.: {\it The geometric reflection principle in several complex variables: a survey.} 
Compl. Var. and Ellipt. Equat. {\bf 54} (2009), 223-241.

\bibitem{DG} Dujardin, R., Guedj, V.: {\it Geometric properties of maximal psh functions}, Lecture Notes Math. 2038, 33-52.

\bibitem{Fe} Fern\'andez-P\'erez, A.: {\it On Levi-flat hypersurfaces with generic real singular set.} J. Geom. Anal. {\bf 23} (2013),
2020-2033. 

\bibitem{FeLe} Fern\'andez-P\'erez, A., Lebl, J.: {\it Global and local aspects of Levi-flat hypersurfaces}.
Publ. Mat. IMPA, Rio de Janeiro, 2015. x+65 pp.

\bibitem{FS} Fornaess, J.~E., Sibony, N.: {\it Riemann surface laminations with singularities}, J.Geom. Anal. {\bf 18} (2008), 400-442.


\bibitem{Le} Lebl, J.: { \it Singular set of a Levi-flat hypersurface is Levi-flat.} Math. Ann.
{\bf 355} (2013), 1177-1199.

\bibitem{Na} Narasimhan, R.  {\it Introduction to the theory of analytic spaces.} Lecture Notes (New York), vol. 25, Springer-Verlag, New York, 1966.


\bibitem{SS} Shafikov, R., Sukhov, A., {\it Germs of singular Levi-flat hypersurfaces and holomorphic foliations}, Comment. Math. Helv. {\bf 90} (2015), 479-502.




\end{thebibliography}
\end{document}